\documentclass[a4paper,10pt]{amsart}
\usepackage{amsmath}
\usepackage{amsthm}
\usepackage[hyphens]{url}
\usepackage{hyperref}
\usepackage[mathscr]{euscript}

\newtheorem{thm}{Theorem}
\newtheorem{lem}{Lemma}
\newtheorem{prop}{Proposition}
\newtheorem*{prop*}{Proposition}
\def\R{\mathbb{R}}
\def\L{\mathfrak{L}}

\def\C{\mathbb{C}}

\def\half{\frac{1}{2}}
\def\supp{\text{supp}}

\title{Large Positive and Negative Values of Hardy's $Z$-Function}

\begin{document}

\author[K. Mahatab]{Kamalakshya Mahatab}
\address{Department of  Mathematical Sciences, Norwegian University of Science and Technology, NO-7491 Trondheim, Norway}
\email{accessing.infinity@gmail.com, \. kamalakshya.mahatab@ntnu.no }

\thanks{The author is supported by Grant 227768 of the Research Council of Norway}


\begin{abstract}
Let $Z(t):=\zeta\left(\half+it\right)\chi^{-\half}\left(\half+it\right)$ be Hardy's function, 
where the Riemann zeta function $\zeta(s)$ has the functional equation $\zeta(s)=\chi(s)\zeta(1-s)$.
We prove that for any $\epsilon>0$,
\begin{align*}
&\quad\max_{T^{3/4}\leq t\leq T} Z(t) \gg \exp\left(\left(\frac{1}{2}-\epsilon\right)\sqrt{\frac{\log T\log\log\log T}{\log\log T}}\right)\\
\text{ and }& \max_{T^{3/4}\leq t\leq T}- Z(t) \gg \exp\left(\left(\frac{1}{2}-\epsilon\right)\sqrt{\frac{\log T\log\log\log T}{\log\log T}}\right).
\end{align*}
\end{abstract}

\maketitle

\section{Introduction}
The Riemann zeta function, $\zeta(s)$, is defined as 
\[\sum_{n=1}^{\infty}\frac{1}{n^s} \quad \text{ for } \Re(s)>1.\]
Further, it has an analytic continuation to the rest of $\C$, except for a simple pole at $s=1$. The analytic continuation of $\zeta(s)$ can be done by the functional equation \cite{tit}:
\[\zeta(s)=\chi(s)\zeta(1-s) \quad \text{ for } s\in \C,\]
where 
\[\chi(s)=\frac{\Gamma\big(\half(1-s)\big)}{\Gamma\big(\half s\big)}\pi^{s-\half}.\]
Properties of $\zeta(s)$ can be used to prove the Prime Number Theorem (PNT). The PNT states that \cite{tenen}
\[\sum_{p\leq x}\Lambda(p)=x+ \Delta(x),\]
where
\[\Lambda(n)=\begin{cases}
              \log p, &\quad \text{ if } \quad n=p^r, r\geq 1,\\
              0 &\quad \text{ otherwise},
             \end{cases}
\]
and 
\[ \Delta(x)=o(x).\]
This implies
\[\#\{ \text{primes} \leq x\}\sim \frac{x}{\log x}.\]
The above formula is a consequence of the fact that $\zeta(s)$ has no zeros on the line $\Re(s)=1$. The best known upper bound for $\Delta(x)$
is due to Vinogradov and Korobov \cite[Chapter 6]{iv4}:
\[\Delta(x)=O\left(x\exp\left(-c\frac{(\log x)^{3/5}}{(\log\log x)^{1/5}}\right)\right),\]
for some constant $c>0$. However, under the Riemann Hypothesis ( RH ), one can show the following stronger result \cite{vonKoch}:
\[\Delta(x)=O(\sqrt x (\log x)^2).\]
Recall that the RH asserts that all the non-trivial zeros $\rho$ of $\zeta(s)$, with $0\leq \Re(\rho)\leq 1$, have $\Re(\rho)=\half$. A simpler problem is to show that there are infinitely many zeros $\rho$ of $\zeta(s)$ such that $\Re(\rho)=\half$.
%However, we may ask the following simpler question: 
% 
% Are there infinitely many zeros $\rho$ of $\zeta(s)$ such that $\Re(\rho)=\half$? 
To prove this claim, Hardy introduced a real valued function $Z(t)$, which is known as Hardy's $Z$-function, defined as follows:
\[Z(t):=\zeta\left(\half+it\right)\chi^{-\half}\left(\half+it\right) \quad \text{for } t \in \R.\]
%where 
%\[\chi(s)=\frac{\Gamma\big(\half(1-s)\big)}{\Gamma\big(\half s\big)}\pi^{s-\half} \quad \text{for } s \in \C.\]
%  The Riemann zeta-function satisfies the following functional equation 
% \[\zeta(s)=\chi(s)\zeta(1-s) \quad s\in \C,\]
% where
% 
%  \[\zeta(s):=\sum_{n=1}^{\infty}\frac{1}{n^s} \text{ for } \Re(s)>1,\]
%  and 
%Hardy's function $Z(t)$ is defined as
$Z(t)$ is a smooth real-valued function in $t$, and 
\begin{equation*}
 |Z(t)|=\left|\zeta\left(\half+it\right)\right|.
\end{equation*}
In other words, the zeros of $\zeta\left(\half+it\right)$ correspond to the zeros of $Z(t)$. Further, Hardy showed that \cite[Lemma~2.3]{iv1}
\begin{equation}\label{eq:hardy}
\int_T^{2T}Z(t) dt\ll T^{3/4} \text{ and } \int_T^{2T}\zeta\left(\half+it\right) dt\gg T.
\end{equation}
This shows that $Z(t)$ must change sign in $[T, 2T]$ and so has a zero in this interval. Since $T$ is arbitrary, $Z(t)$, and hence $\zeta(\half+it)$ must have infinitely many zeros.

To simplify notations, define
\begin{align*}
Z^+(t)&:=\max(0, Z(t))\text{ and }\\
Z^-(t)&:=\max(0, -Z(t)).
\end{align*}
By a theorem of Ramachandra \cite{Ram}
\[\int_T^{2T}\zeta\left(\half+it\right) dt\gg T (\log T)^{1/4}.\]
Combining this with (\ref{eq:hardy}), we may conclude that
\begin{align*}
&\max_{T\leq t\leq 2T} Z^+(t) \gg (\log T)^{1/4}\\
\text{ and }& \max_{T\leq t\leq 2T} Z^-(t) \gg (\log T)^{1/4}.
\end{align*}
If we could obtain reasonable upper bounds for $\int_T^{2T}Z^{2k+1}(t) dt$ for $k\geq 3$ \cite{iv2}, then the above lower bound for $Z^+(t)$ and $Z^-(t)$ can be improved to $(\log T)^{c}, c\geq1/2$. But this is a difficult task, 
given the fact that to show
\[\int_T^{2T}Z^{2k+1}(t) dt=o\left(\int_T^{2T}\zeta^{2k+1}\left(\half + it\right) dt\right)\]
is still an open problem \cite{iv3}.
Using the new exponent pair of Bourgain \cite{Bour}, Ivi\'c \cite{iv2} has proved the following lower bounds for large values of $Z^+(t)$ and $Z^-(t)$: 
\begin{align*}
 &\max_{T\leq t \leq T+ T^{17/110}} Z^+(t) \gg (\log T)^{1/4}, \\
 &\max_{T\leq t \leq T+ T^{17/110}} Z^-(t) \gg (\log T)^{1/4}.
\end{align*}
On the other hand, Balasubramanian and Ramachandra \cite{bala, BaR} proved that there exists a constant $B (\sim 0.530)$ such that
\[\max_{T\leq t\leq T+ H}\left|\zeta\left(\half+it\right)\right|\geq\exp\left(B\sqrt{\frac{\log H}{\log_2 H}}\right) \text{ for } \log_2T\ll H\leq T.\]
In the above inequality $\log\log H$ is denoted by $\log_2 H$. In general we will denote $\underbrace{\log\ldots\log}_{k\text{ times }} t$ by $\log_k t$.
The result of Balasubramanian and Ramachandra suggests that either $Z^+(t)$ or $Z^-(t)$ is bigger than $\exp\left(B\sqrt{\frac{\log H}{\log_2 H}}\right)$, but we do not know which one is big. 
The lower bound for $\zeta\left(\half+it\right)$ has been improved by several authors using the resonance method while allowing $t$ to vary on a larger interval. 
Soundararajan \cite{So} proved that
\[\max_{T\leq t\leq 2T}\left|\zeta\left(\half+it\right)\right|\geq\exp\left((1+o(1))\sqrt{\frac{\log T}{\log_2 T}}\right).\]
Later, Bondarenko and Seip \cite{BS1, BS2} improved this bound significantly
\[\max_{0\leq t\leq T}\left|\zeta\left(\half+it\right)\right|\geq\exp\left((1+o(1))\sqrt{\frac{\log T\log_3 T}{\log_2 T}}\right).\]
Recently, Bret\'eche and Tenenbaum \cite{dlb} optimized the constant in \cite{BS2} to
\[\max_{0\leq t\leq T}\left|\zeta\left(\half+it\right)\right|\geq\exp\left((\sqrt 2+o(1))\sqrt{\frac{\log T\log_3 T}{\log_2 T}}\right).\]
An important ingredient in the resonance method is the following Dirichlet polynomial approximation of $\zeta\left(\half+it\right)$:
\[\zeta\left(\half+it\right)=\sum_{n\leq T}n^{-\half+it}-\frac{T^{\half-it}}{\half-it}+O\left(T^{-\half}\right), \quad |t|\leq T.\]
Further, a weak form of the Riemann-Siegel formula \cite{iv1} asserts that 
\[Z(t)=2\sum_{n\leq\sqrt{t/2\pi}}\frac{1}{\sqrt n}\cos\left(t\log \frac{\sqrt{t/(2\pi)}}{n}-\frac{t}{2}-\frac{\pi}{8}\right)+O\left(\frac{1}{t^{1/4}}\right).\]
In this article, we combine the resonator constructed by Bondarenko and Seip in \cite{BS1} with the above approximation formula for $Z(t)$ to prove
\begin{thm}\label{thm:main1} For any arbitrarily small $\epsilon>0$ and for sufficiently large $T$,
\begin{align*}
&\text{(A)} &\quad\max_{T^{3/4}\leq t\leq T} Z^+(t) \gg \exp\left(\left(\frac{1}{2}-\epsilon\right)\sqrt{\frac{\log T\log_3 T}{\log_2T}}\right),\\
&\text{(B)}  &\text{ and } \max_{T^{3/4}\leq t\leq T} Z^-(t) \gg \exp\left(\left(\frac{1}{2}-\epsilon\right)\sqrt{\frac{\log T\log_3 T}{\log_2T}}\right).
\end{align*}
\end{thm}

Our proof is based on the following observation.
\begin{prop*}
Suppose we could find a non-negative function $K(t)$ such that
\[A(T)\int_{T^{3/4}}^T K(t) dt\ll \int_{T^{3/4}}^T \zeta\left(\half+it\right)K(t)dt \]
\[\text{ and } \int_{T^{3/4}}^T Z(t)K(t)dt=o\left(A(T)\int_{T^{3/4}}^T K(t)dt\right),\]
then 
\[\max_{T^{3/4}\leq t\leq T} Z^+(t),\max_{T^{3/4}\leq t\leq T} Z^-(t)\gg A(T).\]
\end{prop*}
\begin{proof}
To show the above, note that there exists a constant $C_1>0$ such that
\[C_1A(T)\int_{T^{3/4}}^T K(t) dt\leq\int_{T^{3/4}}^T |Z(t)|K(t)dt.\]
Assume that our claim is not true and $\max_{T^{3/4}\leq t\leq T} Z^-(t)\leq C_1A(T)/3.$ Then
\begin{align*}
&\int_{T^{3/4}}^T Z(t)K(t)dt=\int_{T^{3/4}}^T |Z(t)|K(t)dt-\int_{T^{3/4}}^T 2Z^-(t)K(t)dt\\
&\geq\frac{C_1A(T)}{3}\int_{T^{3/4}}^T K(t) dt,
\end{align*}
which contradicts to the fact that $\int_{T^{3/4}}^T Z(t)K(t)dt$ is $o\left(A(T)\int_{T^{3/4}}^T K(t)dt\right)$. This proves our claim for $Z^-(t)$. Similarly we can argue for $Z^+(t)$. 
\end{proof} 
We may also note that the lower bound of $\zeta\left(\half+it\right)$ in \cite{dlb} is the optimal bound that can be 
obtained using the resonance method and the gcd sum technique, while the lower bounds we obtain for $Z^+(t)$ and $Z^-(t)$ 
are not necessarily optimal. There are some technical difficulties in our proof that do not allow us to improve our result. In the course of the 
proof we will observe that if we could improve the upper bound of $\sum_{m\in \mathcal{M}'}r(m)$ in Lemma~\ref{lem:sum_r} to $T^\epsilon \sqrt L$, then we can improve the lower bounds 
in Theorem~\ref{thm:main1} to
$\exp\left(\left(1-\epsilon\right)\sqrt{\frac{\log T\log_3 T}{\log_2T}}\right)$. We will explain the notations $r(m)$ and $\mathcal{M}'$ in Section~\ref{sec:resonator}. 
Further, if we could find an optimal upper bound for the 4-th moment of $R(t)$ ( where $R(t)$ is as defined in \cite{dlb} ), then we could
improve the bound to $\exp\left(\left(\sqrt 2-\epsilon\right)\sqrt{\frac{\log T\log_3 T}{\log_2T}}\right)$.

We can also modify our proof of Theorem~\ref{thm:main1} by using the resonator defined by Soundararajan \cite{So} to prove a weaker lower bound but with a better localization of $t$.

\begin{thm}\label{thm:main2}
For any $\epsilon>0$ and for sufficiently large $T$,
\begin{align*}
&\text{(A)} &\quad\max_{T\leq t\leq 2T} Z^+(t) \gg \exp\left(\left(\frac{1}{2}-\epsilon\right)\sqrt{\frac{\log T}{\log_2T}}\right),\\
&\text{(B)}  &\text{ and } \max_{T\leq t\leq 2T} Z^-(t) \gg \exp\left(\left(\frac{1}{2}-\epsilon\right)\sqrt{\frac{\log T}{\log_2T}}\right).
\end{align*} 
\end{thm}
We will skip the proof of Theorem~\ref{thm:main2} as it is similar to the proof of Theorem~\ref{thm:main1}.

%We would also like to note that from the proof of Theorem~\ref{thm:main1}, 
It is possible to compute lower bounds for the Lebesgue measures of the sets where $Z^+(t)$ and $Z^-(t)$ attain the bounds given in Theorem~\ref{thm:main1} and Theorem~\ref{thm:main2}. 
But these lower bounds are much weaker in compare to the bounds obtained in \cite{GoIv} and \cite{iv2}. As the gain from these computations are not significant and to keep this article short, 
we will not carry out this task. 

In Section~\ref{sec:resonator}, we will go through the notations from \cite{BS1} to define the resonator $R(t)$ and state some related results. In Section~\ref{sec:hardy}, 
we will estimate an integral involving $Z(t)$ and $R(t)$. This will be used in Section~\ref{sec:main_thm} to prove Theorem~\ref{thm:main1}.

\section{Construction of The Resonator}\label{sec:resonator}
The resonator $R(t)$ constructed by Bondarenko and Seip \cite{BS1} has the form of a Dirichlet polynomial:
\begin{equation}\label{eq:res}  R(t)=\sum_{m\in \mathcal{M}'} r(m) m^{-it}. \end{equation}
To define $r(m)$ and $\mathcal M'$, we need the following notations.

Let  $\gamma=1-3\epsilon$, where $\epsilon>0$ is arbitrarily small.
Let $P$ be the set of primes in the interval $( e\log N\log_2N,\quad \log N \exp((\log_2N)^\gamma)\log_2N]$. Define
\[f(p):=\sqrt{\frac{\log N\log_2N}{\log_3 N}}\frac{1}{\sqrt p (\log p -\log_2N-\log_3N)},\]
for $p\in P$ and $0$ on other primes. We assume that $f$ is supported on square-free integers and extend the definition of $f(n)$ as a multiplicative function.   
For a fixed $1<a<\frac{1}{1-\epsilon}$, let $M_k$ be the set of integers having at least $\frac{a\log N}{k^2\log_3N}$ prime divisors in $P_k$, and let
\[ \mathcal{M}:=\supp(f)\setminus\bigcup_{k=1}^{[(\log_2N)^{\gamma}]}M_k.\]
Set $N=[T^{1/4}]$.
Let $\mathcal{J}$ be the set of integers $j$ such that
\[ \Big[(1+T^{-1})^j,(1+T^{-1})^{j+1}\Big)\bigcap \mathcal{M} \neq \emptyset,  \]
and let $m_j$ be the minimum of  $\big[(1+T^{-1})^j,(1+T^{-1})^{j+1}\big)\bigcap \mathcal{M}$ for $j$ in $\mathcal{J}$. Finally, we define
\[ \mathcal{M}':= \big \{ m_j: \ j\in \mathcal{J} \big\}\]
and 
\[ r(m_j):= \left(\sum_{n\in \mathcal{M}, (1-T^{-1})^{j-1} \le n \le (1+T^{-1})^{j+2}} f(n)^2\right)^{1/2} \] 
for every $m_j$ in $\mathcal{M}'$. 

We will also denote $\L:=\sum_{n\in \mathcal{M}} f(n)^2,$
and $\Phi(t):=e^{-t^2}$.

Now we state some results from \cite{BS1}.
\begin{lem}[Lemma~2 of \cite{BS1}]\label{lem:size_M} For large $N$,
 \[|\mathcal{M}'|\leq|\mathcal{M}|\leq N.\]
\end{lem}
% \begin{proof}
% Follows from the proof of Lemma~2 of \cite{BS1}. 
% \end{proof}

\begin{lem}\label{lem:sum_r} The sum of the Dirichlet coefficients of the resonator $R(t)$ has the following upper bound
 \[\sum_{m\in \mathcal{M}'}r(m)\ll T^{1/8}(\log T)\sqrt \L.\]
\end{lem}
\begin{proof}
By Cauchy--Schwarz inequality 
 \begin{align}\label{1}
 \left(\sum_{m\in \mathcal{M}'}r(m)\right)^2\leq |\mathcal{M}'|\sum_{m\in \mathcal{M}'}r(m)^2. %\ll N (\log T)^2\L.
 \end{align}
In \cite{BS1} (see proof of Theorem~1), it has been proved that
\[\sum_{m\in \mathcal{M}'}r(m)^2\ll (\log T)^2 \L,\]
and by Lemma~\ref{lem:size_M}, $\mathcal{M}'\leq T^{1/4}$.
Substituting these bounds in (\ref{1}) proves the lemma. 
\end{proof}

\begin{lem}\label{lem:int_R} For large $T$, 
 \[\int_{T^{3/4}}^{T}|R(t)|^2 \Phi\left(\frac{t\log T}{T}\right) dt\ll T(\log T)^3 \L. \]
\end{lem}
\begin{proof}
 See (22) of \cite{BS1}.
\end{proof}

 \begin{prop}\label{prop:1}For an arbitrarily small $\epsilon>0$, we have
  \begin{align*} 
\int_{T^{3/4}}^{T} \zeta\left(\half+i t\right)|R(t)|^2 \Phi\left(\frac{t\log T}{T}\right) dt 
 \gg \L \ T\exp\left(\left(\frac{1}{2}-\epsilon\right)\sqrt{\frac{\log T\log_3 T}{\log_2T}}\right).
\end{align*}
\end{prop}
\begin{proof}
See (25) of \cite{BS1}. 
\end{proof}

\section{Resonator and Hardy's Function}\label{sec:hardy}
Recall that Hardy's function $Z(t)$ has the following approximation formula ((2.3) \cite{iv1}):
\begin{align}\label{eq:z}
\notag
Z(t)&=2\sum_{n\leq \sqrt{t/2\pi}}\frac{1}{\sqrt n}\cos\left(F_n(t)\right) +O\left(t^{-1/4}\right) \\
&=2\Re\left(\sum_{n\leq \sqrt{t/2\pi}}\frac{1}{\sqrt n}\exp\left(iF_n(t)\right)\right)+O\left(t^{-1/4}\right), 
\end{align}
where
\begin{align*}
 F_n(t)=t\log \frac{\sqrt{t/2\pi}}{n}-\frac{t}{n}-\frac{\pi}{8}.
\end{align*}
We also need the second derivative test to estimate certain integrals involving $Z(t)$: 
\begin{lem}[Lemma~2.3 of \cite{iv1}]\label{lem:sec_der}
 Let $F(x)$ be a real twice differentiable function for $a\leq x\leq b$ such that $|F^{''}(x)|\geq m \ (>0)$. Let $G(x)$ be a positive monotonic function such that $|G(x)|\leq G$
 for $x\in [a, b]$. Then
 \[\int_a^bG(x)\exp(iF(x)) dx\leq \frac{8 G}{\sqrt m}.\]
\end{lem}
% \begin{proof}
% See Lemma~2.3 of \cite{iv1}. 
% \end{proof}

Using the above approximation formula for $Z(t)$ and the second derivative test, we prove the following proposition.
\begin{prop}\label{prop:2} Let $R(t)$ be defined as in Section~\ref{sec:resonator}. Then for large $T$,
 \begin{align*}
 \int_{T^{3/4}}^{T} Z(t) |R(t)|^2\Phi\left(\frac{t\log T}{T}\right) dt
 \ll \L \ T(\log T)^2. 
 \end{align*}
\end{prop}
\begin{proof}
%  We exchange the integrals in left as follows:
We plug in the formulas for $Z(t)$ and $R(t)$ from (\ref{eq:res}) and (\ref{eq:z}) respectively in the above integral, and then exchange the sums and the integral to get
 \begin{align*}
&J:=\int_{T^{3/4}}^{T} Z(t) |R(t)|^2\Phi\left(\frac{t\log T}{T}\right) dt \\
&\ll\quad\sum_{m, n\in \mathcal{M}'}r(m)r(n) 
\sum_{k\leq \sqrt{\frac{T}{2\pi}}}\frac{1}{\sqrt k}\left|\int_{\max{(2\pi k^2, \ T^{3/4})}}^{T}\exp\left(iF_k(t)+it\log (m/n)\right)\Phi\left(\frac{t\log T}{T}\right)dt\right|.
\end{align*}

To apply the second derivative test, observe that
\[\frac{d^2}{dt^2}(F_k(t)+it\log (m/n))\gg \frac{1}{t}\gg \frac{1}{T},\]
when $T^{3/4}\leq t\leq T$. So by Lemma~\ref{lem:sum_r} and Lemma~\ref{lem:sec_der}, we have
\begin{align*}
 J &\ll \sqrt{T}\sum_{m, n\in \mathcal{M}'}r(m)r(n)\sum_{k\leq \sqrt{\frac{T}{2\pi}}}\frac{1}{\sqrt k}\\
 &\ll T^{3/4}\left(\sum_{m\in \mathcal{M}'}r(m)\right)^2 \ll \L \ T(\log T)^2. 
\end{align*}
\end{proof}

\section{Proof of Theorem~\ref{thm:main1}}\label{sec:main_thm}
We prove Theorem~\ref{thm:main1} by comparing Proposition~\ref{prop:1} and Proposition~\ref{prop:2}.
We will proceed by the method of contradiction. So assume that
 \begin{equation}\label{asump}Z^-(t)\leq C_1\exp\left(\left(\frac{1}{2}-\epsilon\right)\sqrt{\frac{\log T\log_3 T}{\log_2T}}\right),\end{equation}
 for all $t\in[T^{3/4}, T]$ and for some $C_1>0$. Let 
\[J_1:=\int_{T^{3/4}}^{T} \left|\zeta\left(\half+i t\right)\right||R(t)|^2 \Phi\left(\frac{t\log T}{T}\right) dt. \]
Then by Proposition~\ref{prop:1}, for any $\epsilon>0$,
\begin{equation}\label{j1}
 J_1\gg \L T\exp\left(\left(\frac{1}{2}-\frac{\epsilon}{2}\right)\sqrt{\frac{\log T\log_3 T}{\log_2T}}\right).
\end{equation}
Define 
 \[J_2=J_1 - 2\int_{T^{3/4}}^{T} Z^-( t)|R(t)|^2 \Phi\left(\frac{t\log T}{T}\right) dt.\]
 We will bound $J_2$ from below using assumption (\ref{asump}) and Lemma~\ref{lem:int_R} as follows
 \begin{equation}\label{J2}
 J_2\gg \L T\exp\left(\left(\frac{1}{2}-\frac{\epsilon}{2}\right)\sqrt{\frac{\log T\log_3 T}{\log_2T}}\right),
 \end{equation}
 as 
\begin{align*}
&\int_{T^{3/4}}^{T} Z^-( t)|R(t)|^2 \Phi\left(\frac{t\log T}{T}\right) dt\\
&\ll \exp\left(\left(\frac{1}{2}-\epsilon\right)\sqrt{\frac{\log T\log_3 T}{\log_2T}}\right)\int_{T^{3/4}}^{T} |R(t)|^2 \Phi\left(\frac{t\log T}{T}\right) dt\\
&\ll \L T \exp\left(\left(\frac{1}{2}-\frac{2\epsilon}{3}\right)\sqrt{\frac{\log T\log_3 T}{\log_2T}}\right).
\end{align*}

Since $|Z(t)|=Z(t)+2Z^-(t)$, 
\[J_2\leq  \int_{T^{3/4}}^{T} Z( t)|R(t)|^2 \Phi\left(\frac{t\log T}{T}\right) dt.\] 
From Proposition~\ref{prop:2} we  get
\begin{align*}
J_2\ll \L \ T(\log T)^2,
\end{align*}
which contradicts  (\ref{J2}). So our assumption (\ref{asump}) for $Z^{-}(t)$ is false. This proves (B) of Theorem~\ref{thm:main1}, and the proof of (A) is similar.

\end{document}